\documentclass[11pt,reqno]{amsart}
\usepackage{amsmath, amsthm, amssymb, bbm, color}
\usepackage[hmargin={1.3in, 1.3in}, vmargin={1.3in, 1.3in}]{geometry}
\usepackage[breaklinks=true]{hyperref}
\allowdisplaybreaks

\theoremstyle{plain}
\newtheorem{theorem}{Theorem}[section]
\newtheorem{proposition}[theorem]{Proposition}

\newtheorem{lemma}[theorem]{Lemma}

\newtheorem{utheorem}{\textrm{\textbf{Theorem}}}

\theoremstyle{definition}
\newtheorem{definition}[theorem]{Definition}
\newtheorem{remark}[theorem]{Remark}
\newtheorem{example}[theorem]{Example}

\numberwithin{equation}{section}

\newcommand{\F}{\mathbb F}
\newcommand{\mc}{\mathcal}
\newcommand{\mbm}{\mathbbm}
\newcommand{\mfm}{\mathfrak{m}}

\renewcommand{\ge}{\geqslant}
\renewcommand{\leq}{\leqslant}

\begin{document}

\title[A tight bound for affine-linearity, via universal ballot
matrices]{A tight bound for affine-linearity, via\\ universal ballot
matrices}

\author{Apoorva Khare}
\address[A.~Khare]{Indian Institute of Science, Bangalore 560012, India}
\email{\tt khare@iisc.ac.in}

\author{Ashwin Sah}
\address[A.~Sah]{}
\email{\tt ashwinsahmath@gmail.com}



\date{\today}

\begin{abstract}
Based on work with Greenfeld and with Ziegler, Tao showed a concatenation
result that if a map $f : \mathbb{F}^2 \to \mathbb{F}$ is affine-linear
on every  line parallel to the coordinate axes, and on all lines with a
fixed nonzero slope (where the field $\mathbb{F}$ has size $> 2$), then
$f$ is affine-linear on $\mathbb{F}^2$. We extend this from
$\mathbb{F}^2$ to $\mathbb{F}^n$ and obtain a tight minimum number of
additional lines needed -- $N = \binom{n}{\lfloor n/2 \rfloor}$ -- for
every field $\mathbb{F}$ with $3 \leqslant n < |\mathbb{F}|$.

The proof is constructive and shows a stronger result: the existence of a
universal family of $0$-$1$ matrices of size $\binom{n}{k} \times
\binom{n}{k}$ (one for each pair $0 \leqslant k \leqslant n$), which are
indexed by ballot sets and are unimodular over all unital commutative
rings. We also show a second tightness: of the assumption $n <
|\mathbb{F}|$. Else, there exist multi-affine maps $f$ which are
affine-linear on every line through the origin, but not affine-linear
globally on $\mathbb{F}^n$.
More strongly, we prove this dichotomy -- including the bound of $N$ --
over all integral domains, or Noetherian (e.g.\ finite or Artinian)
rings, or products of these. This yields a novel numerical invariant for
affine-linearity, for every product of Noetherian rings and integral
domains.
\end{abstract}

\maketitle

\section{The main results}

Given an integer $k \ge 1$ and a set $S$, write $[k] := \{ 1, \dots, k
\}$ and let $\binom{S}{k} := \{ S' \subseteq S : |S'| = k \}$.

\subsection{Concatenation theorems}

The results below are in the theme of \textit{concatenation theorems}, as
discussed e.g.\ in~\cite{TZ} for polynomial maps on products of abelian
groups. Here we focus on multi-affine polynomials of several variables.

We begin in greater generality. Let $H \leq G$ and $K$ be abelian groups.
In \cite{TZ}, Tao and Ziegler define a map $P : G \to K$ to be a
\textit{polynomial} of degree $<d$ along $H$ -- for $d \leq 0$ if $P
\equiv 0$, and for $d>0$ if for every $h \in H$ there exists a polynomial
$P_h : G \to K$ of degree $< d-1$ along $H$, such that $P(g+h) = P(g) +
P_h(g)$ for all $g \in G$.

In particular, polynomials of degree $<1$ are constant maps along
$H$-cosets in $G$: $P(g+H) = P(g)$ for all $g \in G$. Now Tao--Ziegler
proved in general:

\begin{theorem}[{\cite[Proposition~1.2]{TZ}}]
Suppose $H_1, H_2 \leq G$ and $K$ are abelian groups, and $d_1, d_2 \in
\mathbb{Z}$. If $P : G \to K$ is a polynomial of degree $< d_i$ along
$H_i$ for $i=1,2$, then $P$ is a polynomial of degree $< d_1 + d_2 - 1$
along $H_1 + H_2$.
\end{theorem}

When $d_1 = d_2 = 2$, with additional data one can do better than $d_1 +
d_2 - 1 = 3$, as follows:
given abelian groups $H_1, H_2 \leq G$ and $K$, polynomials that are of
degree $<2$ along $H_1$ and $H_2$ satisfy
\[
P(g + n_1 h_1 + n_2 h_2) = P(g) + n_1 c_{h_1}(g) + n_2 c_{h_2}(g), \qquad
\forall h_1 \in H_1, h_2 \in H_2, \, n_1, n_2 \in \mathbb{Z}, g \in G
\]
for some maps $c_{h_i} : G/H_i \to K$, among other properties. These
resemble multi-affine maps, i.e.\ of degree $<2$ rather than the
Tao--Ziegler bound of $<3$. This was further explored by Greenfeld--Tao
in their study of the periodic tiling conjecture~\cite{GT}. In
particular, in his blog discussing~\cite{GT}, Tao states the following
result, which we have stripped of various technical but unnecessary
assumptions (such as the field being $\mathbb{R}$ or of odd prime order):

\begin{theorem}[\cite{T}]\label{Ttao}
Fix a field $\F$ with characteristic not $2$, and let $f : \F^2 \to \F$
be affine-linear when restricted to every horizontal, diagonal, and
anti-diagonal affine line (i.e.\ of slopes $0, 1, -1$ respectively). Then
$f$ is affine-linear on $\F^2$.
\end{theorem}

As the characteristic is not $2$, one can replace the (anti-)diagonal
directions by any basis $e_1, e_2$ of $\F^2$ (avoiding the horizontal
direction). Thus, Theorem~\ref{Ttao} can be reformulated via this
change-of-basis, to say:

\textit{Suppose $f : \F^2 \to \F$ is affine-linear on every affine line
$m_0 + \F e_1$ and $m_0 + \F e_2$, for all $m_0 \in \F^2$. If $f$ is
moreover affine-linear on every affine line $m_0 + \F (v_1, v_2)$ for
some fixed $v_1, v_2 \in \F^\times$, then it is affine-linear on all of
$\F^2$.}

These hypotheses were significantly weakened by the first-named author
with Tikaradze~\cite{KT}, to show that one needs not all affine lines
$m_0 + \F (v_1, v_2)$, but just one line -- with $m_0 = (0,0)$. This was
further generalized to all dimensions, and beyond fields:

\begin{theorem}[{\cite[Theorem~1.3]{KT}}]\label{TKT1}
Suppose $R$ is a unital commutative ring, $n > 1$, and $e_1, \dots, e_n$
constitute the standard basis of $R^n$. Also suppose $1, 1+1, \dots, n
\in R$ are non-zerodivisors/regular elements.\footnote{We adopt the
convention that $0$ is also a zerodivisor.} Suppose $f : R^n \to R$ is
any map such that $f$ restricted to $m_0 + R e_i$ is affine-linear, for
all $m_0 \in R^n$ and $1 \leq i \leq n$.

Then there exist $2^n - (n+1)$ vectors ${\bf v}_i \in R^n$ such that if
$f$ restricted to each $R \cdot {\bf v}_i$ is affine-linear, then $f$ is
affine-linear on $R^n$ -- i.e., there exist $a_1, \dots, a_n \in R$ such
that
\[
f (c_1 e_1 + \cdots + c_n e_n) = f(0) + \sum_{i=1}^n a_i c_i, \qquad
\forall c_1, \dots, c_n \in R.
\]
\end{theorem}

For historical completeness, we add that the quest to prove that
``restricted affine-linearity'' implies global affine-linearity has been
a topic studied for over a century. See e.g.\ von Staudt's
work~\cite{vS}, Hartshorne's course-notes~\cite{H} which he based on
Zariski's lectures, and other references in~\cite{KT}.

\subsection{A tight bound: I.~Sufficiency}

Before proceeding, here is a simplifying observation:

\begin{proposition}\label{Pmultiaffine}
Suppose $R$ is any unital commutative ring, $n \ge 1$, and $f : R^n \to
R$ any map. Then $f$ is affine-linear on $m_0 + R \cdot e_i$ for all $m_0
\in R^n$ and $1 \leq i \leq n$, if and only if $f$ is a multi-affine
polynomial on $R^n$:
\begin{equation}\label{EaJ}
f({\bf x}) = \sum_{J \subseteq [n]} a_J {\bf x}^J \qquad \text{on } R^n,
\end{equation}
for some scalars $a_J \in R$, and where ${\bf x}^J := \prod_{j \in J}
x_j$. 
\end{proposition}

\begin{remark}
As we are not sure Proposition~\ref{Pmultiaffine} was proved over general
rings, a short proof is included below for completeness. 
Here we add that this result (even for the special case of finite fields)
connects our paper to important results in computational complexity. The
result ``MIP = NEXP'' was shown by
Babai--Fortnow--Lund~\cite{BFL}, and deals with functions that agree on a
proportion of the domain with a multi-affine map (which is termed
``multilinear'' in the complexity literature, i.e., linear in each
variable upon fixing values for all others). Trying to quantify this
involves exploring when a function $f$ satisfying some form of an
affine-linearity functional equation for a large fraction of inputs in
$\F_q^n \times \F_q^n$, implies that $f$ matches a single multi-affine
map on (another) large fraction of inputs in $\F_q^n$. The PCP theorem is
a scaling down of the ``MIP = NEXP'' result to NP and is obtained by an
extension of the Babai--Fortnow--Lund theorem from multi-affine functions
to low-degree polynomials.
\end{remark}

Given Proposition~\ref{Pmultiaffine}, we will replace the above
assumptions in Tao's result (Theorem~\ref{Ttao}) -- when used below --
equivalently by ``multi-affine'' for brevity. Now it is natural to ask if
the bound of $2^n - (n+1)$ in Theorem~\ref{TKT1} can be reduced for $n
\ge 3$ (for $n \leq 2$ it is $\leq 1$). This was indeed carried out
in~\cite{KT} in various restricted settings:

\begin{theorem}[{\cite[Theorem~3.6]{KT}}]\label{TKT}
Fix an integer $n \ge 3$. Suppose $R$ is any of:
(a)~a finite integral domain (i.e., finite field) of size at least
$2^{n-1}+1$,
(b)~an infinite integral domain, or
(c)~a unital commutative ring in which $2 = 1+1, 3 = 1+1+1, \dots, p(n)$
(the $n$th prime integer) are non-zerodivisors. Then there exist
\[
N := \binom{n}{\lfloor n/2 \rfloor}
\]
directions ${\bf v}_1, \dots, {\bf v}_N \in R^n$ such that if $f : R^n
\to R$ is multi-affine and $f|_{R {\bf v}_i}$ is affine-linear for all
$i$, then $f$ is affine-linear on $R^n$. Fewer than $N$ directions will
not suffice.
\end{theorem}

Note by Stirling's formula that $N \sim 2^{n + \frac{1}{2}} / \sqrt{\pi
n}$, so $N = o(2^n - (n+1))$. Theorem~\ref{TKT} was proved using
multiplicative analogues of Sidon-sets, aka $B_h$-sets. That is, the
proof of Theorem~\ref{TKT} uses the existence of an $n$-element subset $S
\subset R$ satisfying:
\begin{equation}\label{Esidon}
for\ each\ fixed\ 2 \leq h \leq n,\ the\ subsets\ in\ \binom{S}{h} \
have\ pairwise\ distinct\ products.
\end{equation}

In a sense, this means the ground ring is ``large enough'' that such sets
exist in $R$, whence there exist enough ``generic'' directions ${\bf
v}_i$ that help show global affine-linearity.\medskip

Below, our objective is to extend Theorem~\ref{TKT} to every field -- and
in parallel, we also extend it to every finite unital commutative ring.
Note that if $|\F| < N$, then a subset $S$ satisfying~\eqref{Esidon}
necessarily does not exist for $h = \lfloor n/2 \rfloor$. Thus, the
strategy changes: to find not generic but \textit{specific} $N$-many
directions ${\bf v}_1, \dots, {\bf v}_N$ that will make Theorem~\ref{TKT}
work still. Even more challenging (if true, which is not clear e.g.\
from~\cite{KT}) is the existence of such directions over arbitrary rings
$R$ -- individually, or stronger still, all $R$ at once.

In this work, we show that this strongest scenario indeed holds for all
fields (and part of it holds for all rings). We term our universal and
minimal set of $N$ vectors ${\bf v}_i \in \{ 0, 1 \}^n$ as the
\textbf{ballot directions}; these are alluded to in Theorem~\ref{T2}(3)
and explicitly given in Theorem~\ref{thm:main}.

We now present our main results. These require the following notions:

\begin{definition}
Throughout this paper, $R$ denotes a unital commutative ring. Given an
integer $n \ge 2$, define the \textit{$n$th Vandermonde polynomial} to be
\begin{equation}\label{Epn}
p_n({\bf x}) \ := \ \prod_{j=1}^n x_j \cdot \prod_{1 \leq j < k \leq n}
(x_j - x_k) \quad \in R[x_1, \dots, x_n],
\end{equation}
and the associated \textit{$n$th Vandermonde ideal} $I_{p_n}$ to be
generated by the image of $p_n$:
\begin{equation}
I_{p_n} = I_{p_n}(R) := \langle {\rm im}(p_n) \rangle =
\left\langle \{ p_n({\bf r}) : {\bf r} \in R^n \} \right\rangle.
\end{equation}
Finally, we say (a unital commutative ring) $R$ is
\textit{$n$-Vandermonde full} if ${\rm Ann}_R (I_{p_n}) = 0$.
\end{definition}

For instance, ${\rm Ann}_R(I_{p_n}) = 0$ if (but not only if) $p_n({\bf
r})$ is a non-zerodivisor for some tuple ${\bf r} \in R^n$. Also, we
attach the name ``Vandermonde'' to the polynomial $p_n$ (and hence to the
ideal $I_{p_n}$) since $p_n({\bf r})$ is the determinant of the matrix
\[
\begin{pmatrix} 1 & 0 & 0 & \cdots & 0 \\
1 & r_1 & r_1^2 & \cdots & r_1^n \\
1 & r_2 & r_2^2 & \cdots & r_2^n \\
\vdots & \vdots & \vdots & \ddots & \vdots\\
1 & r_n & r_n^2 & \cdots & r_n^n \\
\end{pmatrix} \in R^{(n+1) \times (n+1)}.
\]

Our first main result presents a sufficient condition (a)~under which a
multi-affine map $: R^n \to R$ must be affine-linear on $R^n$ if it is
affine along all (ballot) directions; and
(b)~which is broad enough to subsume the above theorems.

\begin{utheorem}\label{T1}
Suppose $n \ge 3$. If a ring $R$ is $n$-Vandermonde full, then
Theorem~\ref{TKT} holds over $R$. In particular, the bound of $N$
directions is sharp.
\end{utheorem}

This result subsumes Theorem~\ref{TKT}, as well as Theorem~\ref{TKT1}
(given that $p_n(1, 1+1, \dots, n)$ is a non-zerodivisor in $R$, so ${\rm
Ann}_R(I_{p_n}) = 0$). This is because if $n \ge 3$ then $2^n - (n+1) \ge
N = \binom{n}{\lfloor n/2 \rfloor}$.

\begin{remark}\label{Rfield}
Theorem~\ref{T1} addresses the case of all $n \ge 3$. For $n=2$, the
sharp bound was $1 < N = 2$, by Theorem~\ref{TKT1}.
\end{remark}

\subsection{A tight bound: II.~Necessity for products of Kasch-type
rings}\label{Skasch}

If $R = \F$ is a field -- or even an integral domain -- then it is easy
to see that $R$ being $n$-Vandermonde full simply means $|R| > n$. Thus,
Theorem~\ref{T1} extends Theorem~\ref{TKT} from domains of size $>
2^{n-1}$ to all domains of size $>n$ (noting that finite integral domains
are fields). The reader may wonder if one can reduce the size even
further. The next result explains why one cannot, so that
Theorem~\ref{T1} is the ``best possible'' along these lines. In fact, it
is best possible not only for $R$ a domain, but for a large class
$\mathfrak{R}$ of rings, which we describe before our next result.

\begin{definition}\label{DR}
As above, $R$ is a unital commutative ring. Denote by $ZD(R)$ its
zerodivisor set -- so the complement $ZD(R)^c$ is multiplicatively
closed. Next, define the \textit{total ring of fractions} of $R$ to be
the localization
\[
Q(R) := (ZD(R)^c)^{-1} \cdot R.
\]
Finally, we say $R$ is a \textit{Kasch ring} \cite[Chapter~8C]{Lam} if
every maximal ideal in $R$ has nonzero annihilator in $R$.
\end{definition}

The family $\mathfrak{R}$ can now be defined:

\begin{definition}\label{Dkasch}
Let $\mathfrak{R}$ consist of all unital commutative rings
$\times_{i \in I} R_i$, where
\begin{enumerate}
\item $I$ is a nonempty index set, and
\item each $R_i$ is a unital commutative ring whose total ring of
fractions $Q(R_i)$ is Kasch.
\end{enumerate}
\end{definition}

Before proceeding further, we give some examples of such rings.

\begin{lemma}\hfill
\begin{enumerate}
\item Every integral domain, hence field, satisfies
Definition~\ref{Dkasch}(2) and hence is in $\mathfrak{R}$.

\item If $(R, \mfm)$ is a local ring with nilpotent maximal ideal $\mfm$,
then $R$ satisfies Definition~\ref{Dkasch}(2) and hence is in
$\mathfrak{R}$.

\item If $k \ge 1$ and rings $R_1, \dots, R_k$ satisfy
Definition~\ref{Dkasch}(2) then so does $R_1 \times \cdots \times R_k$.

\item In particular, $\mathfrak{R}$ contains all Artinian rings (e.g.\
all finite unital commutative rings). More strongly, every Artinian ring
satisfies Definition~\ref{Dkasch}(2).
\end{enumerate}
\end{lemma}

\begin{proof}
The first statement (for integral domains) holds since the only maximal
ideal of $Q(R)$ here is $0$. The second holds because $ZD(R) = \mfm$ and
so $Q(R) = R$, and its unique maximal ideal is $\mfm$. Now if $\mfm^k
\neq 0 = \mfm^{k+1}$, then any $0 \neq r \in \mfm^k$ annihilates $\mfm$.

To show the third statement, note that
$ZD(R)^c = \times_{i=1}^k (R_i \setminus ZD(R_i))$. Therefore,
\[
Q(R) = \left( \times_{i=1}^k (ZD(R_i)^c \right)^{-1} \cdot \times_{i=1}^k
R_i = \times_{i=1}^k (ZD(R_i)^c)^{-1} R_i = \times_{i=1}^k Q(R_i).
\]
Hence every maximal ideal in $Q(R)$ is of the form $\mfm_i \times
(\times_{j \neq i} Q(R_j)) = Q(R) \cdot ({\bf 1}^{i-1}, \mfm_i, {\bf
1}^{n-i})$. Since $Q(R_i)$ is Kasch, suppose $\frac{a_i}{s_i} \mfm_i = 0$
for some $0 \neq a_i \in R_i$ and $s_i \in R_i \setminus ZD(R_i)$. Then
$a_i \cdot \mfm_i \times (\times_{j \neq i} Q(R_j)) = 0$, whence $Q(R)$
is indeed Kasch.

The final assertion follows from (the third and) the second assertion by
the structure theorem for Artinian rings, see e.g.\ \cite[Proposition~8.6
and Theorem~8.7]{AMcD}.
\end{proof}

The reader may wonder if every local ring $(R, \mfm)$ satisfying
Definition~\ref{Dkasch}(2) must have $\mfm$ nilpotent. The following
example shows that is not the case, so that our family $\mathfrak{R}$ is
larger.

\begin{example}\label{Exgemini}
Let $\F$ be any field and set $R := \F[[x]] \times \F$, with
componentwise addition and with (commutative) multiplication given by:
\[
(f(x),r) \cdot (g(x),s) := (f(x)g(x), f(0)s + g(0)r).
\]
The multiplicative identity here is $(1,0)$, so the units are $(f(x),r)$
with $f(0) \neq 0$: we have $(f(x),r)^{-1} = (1/f(x), -r / f(0)^2)$. In
particular, the set of non-units is $\mfm := \{ (f(x),r) : f(0) = 0 \}$.
It is clear that $\mfm$ is an ideal, so it is maximal and $R$ is local.
However, $(x,0) \in \mfm$ is not nilpotent: $(x,0)^m = (x^m,0)$ for all
$m \ge 1$; and yet $(0,1) \cdot \mfm = 0$. \qed
\end{example}

The preceding example is not Artinian, given the decreasing
sequence of ideals $(x) \times \F \supset (x^2) \times \F \supset
\cdots$. However, it is a product of Noetherian rings, hence Noetherian.
Thus it is natural to ask if $\mathfrak{R}$ contains Noetherian rings --
and this is indeed true:

\begin{proposition}\label{Pnoetherian}
Every Noetherian ring satisfies Definition~\ref{Dkasch}(2), and hence
arbitrary products of Noetherian rings are in $\mathfrak{R}$.
\end{proposition}

\noindent (This is proved below.) With these notions and preliminaries,
we can now state our second main result, whose final assertion is for all
rings in $\mathfrak{R}$:

\begin{utheorem}\label{T2}
Fix an integer $n \ge 3$. Then for any unital commutative ring $R$ and
any multi-affine map $f : R^n \to R$, each assertion below implies the
next:
\begin{enumerate}
\item $f$ is affine-linear on all of $R^n$.

\item $f|_{R {\bf v}}$ is affine-linear for all directions ${\bf v} \in
R^n$. That is $f(r \cdot {\bf v}) = f(0) + r m_{\bf v}$ for all $r \in
R$, where the ``slope'' $m_{\bf v} = f({\bf v}) - f(0)$.

\item There exist $N = \binom{n}{\lfloor n/2 \rfloor}$ ``ballot
directions'' ${\bf v}_i$ -- in fact all belonging to $\{ 0, 1 \}^n$,
independent of $R$, and depending only on $n$ -- such that $f|_{R{\bf
v}_i}$ is affine-linear for $1 \leq i \leq N$.
\end{enumerate}

Moreover, the following dichotomy holds:
\begin{itemize}
\item[(a)] If $R$ is $n$-Vandermonde full, then the above assertions are
equivalent: $(3) \implies (1)$. Moreover, $N$ is the best possible bound.

\item[(b)] If $R$ is not $n$-Vandermonde full -- and additionally, $R \in
\mathfrak{R}$, then the assertions are not equivalent. More strongly,
$(2)$~does not imply~$(1)$.
\end{itemize}
\end{utheorem}

Given the line after Remark~\ref{Rfield}, we can reformulate
Theorem~\ref{T2} in the special case of $R = \F$ a field (and $n \ge 3$)
as follows. We have $|\F| > n$ if and only if: a multi-affine polynomial
$f : \F^n \to \F$ being affine-linear on $\F {\bf v}$ for every ${\bf v}
\in \F^n$ (equivalently, for all $0$-$1$ vectors ${\bf v}$, or the even
smaller set of ballot directions), implies $f$ is affine-linear on
$\F^n$.

\subsection{A numerical invariant for affine-linearity}

Our final main result derives from Theorems~\ref{T1} and~\ref{T2} a
numerical invariant $\tau(R)$ for affine-linearity. We first consider
$\F$ an integral domain or a field. In light of Theorem~\ref{T2}, the
first and last paragraphs in Section~\ref{Skasch} reveal an invariant of
$\F$ -- its size $\tau(\F) := |\F|$ -- such that $\tau(\F) > n$ is
equivalent to $\F$ being $n$-Vandermonde full. Now Theorem~\ref{T2}
reveals a dichotomy depending on whether or not $\tau(\F) > n$, about
whether or not a multi-affine map that is affine in every direction (or
just along the ballot directions) is globally affine-linear.

The results in this work extend both this equivalence and this dichotomy,
from fields $\F$ to all rings $R \in \mathfrak{R}$:

\begin{utheorem}\label{T3}
Let $R = \times_{i \in I} R_i$ be in $\mathfrak{R}$, with all $Q(R_i)$
Kasch. Define
\begin{equation}\label{Etau}
\tau (R) = \tau \left( \times_{i \in I} R_i \right) :=
\min_{i \in I} \min_{\mfm_i \in {\rm MaxSpec}\; Q(R_i)}
|R_i / (\mfm_i \cap R_i)|
\end{equation}
(which is possibly infinite).
Now given an integer $n \ge 3$, the following are equivalent:
\begin{enumerate}
\item $R$ is $n$-Vandermonde full.

\item $\tau(R) > n$.

\item If $f : R^n \to R$ is a multi-affine polynomial, and $f|_{R {\bf
v}}$ is affine-linear for (all ballot, or) all directions ${\bf v} \in
R^n$, then $f$ is affine-linear on $R^n$.
\end{enumerate}
\end{utheorem}

In the special case of $R$ a field or even an integral domain, indeed
$\tau(R) = |R|$.

\section{Theorem~\ref{T1}: reduction to universal nonsingular 0-1 matrices}\label{Smccoy}

We address Theorems~\ref{T2} and~\ref{T3} in Section~\ref{SthmB}. In this
section and the next, we prove Theorem~\ref{T1}. We begin by addressing
the characterization of multi-affine maps over any ground ring:

\begin{proof}[Proof of Proposition~\ref{Pmultiaffine}]
The forward direction was proved in \cite[pp.~765--766]{KT}. Conversely,
for all ${\bf m} = (m_1, \dots, m_n) \in R^n$,
$r \in R$, and $1 \leq i \leq n$, we compute:
\[
f({\bf m} + r \cdot e_i) = \sum_{i \not\in J \subseteq [n]} a_J {\bf m}^J
+ (m_i + r) \sum_{i \in J \subseteq [n]} a_J {\bf m}^{J \setminus \{ i 
\}} = f({\bf m}) + r \sum_{i \in J \subseteq [n]} a_J {\bf m}^{J
\setminus \{ i \}},
\]
which is clearly affine-linear on ${\bf m} + R e_i$.
\end{proof}

We next turn to Theorem~\ref{T1}. If $f$ is affine-linear on $R \cdot
{\bf v}$ for some ${\bf v} \in R^n$, then for all $r \in R$,
\begin{equation}\label{Evandermonde}
f(r \cdot {\bf v}) = \sum_{k=0}^n r^k c({\bf v},k), \qquad \text{where }
c({\bf v},k) := \sum_{|J|=k} a_J {\bf v}^J.
\end{equation}
Here $c({\bf v},0) = f(0)$; by affine-linearity, we also have $f(r \cdot
{\bf v}) = f(0) + r \cdot m_{\bf v}$ for some $m_{\bf v} \in R$.
Equating,
\begin{equation}\label{Evan}
r (c({\bf v},1) - m_{\bf v}) + \sum_{k=2}^n r^k c({\bf v},k) =
0, \qquad \forall r \in R.
\end{equation}

Given the hypotheses, specialize this to any $n$-tuple $r = r_1, \dots,
r_n$, and repackage into a system of linear equations. The coefficient
matrix is a Vandermonde matrix with determinant
$\prod_{j=1}^n r_j \cdot \prod_{1 \leq j<k \leq n} (r_j-r_k) = p_n({\bf
r})$. Pre-multiplying by the adjugate matrix, we have that $p_n({\bf r})
c({\bf v},k) = 0$ for all ${\bf r} \in R^n$. Thus $c({\bf v},k)$
annihilates $I_{p_n}(R)$. But $R$ is $n$-Vandermonde full, and so $c({\bf
v},k) = 0$ whenever $2 \leq k \leq n$. Hence $f|_{R \cdot {\bf v}}$ is
affine-linear.

Recall that our goal is to show in~\eqref{EaJ} that $a_J = 0$ for all
$|J| \ge 2$, i.e.\ for all $a_J$ occurring in $c({\bf v},k)$ for $2 \leq
k \leq n$.
We achieve this (in a ``universal'' fashion) in the next section, via a
careful choice of $N$-many ballot directions ${\bf v}_1, \dots, {\bf
v}_N$, as alluded to in Theorem~\ref{T2}(3).

For now, we end by showing the final part of Theorem~\ref{T1}: if $f$ is
affine-linear along $R \cdot {\bf v}_1, \dots, R \cdot {\bf v}_m$, and $m
< N = \binom{n}{\lfloor n/2 \rfloor}$, then this \textit{does not
guarantee} all $a_J = 0$ (for our fixed ground ring $R$).

To see why, let $k = \lfloor n/2 \rfloor$. From above, we have
\[
0 = c({\bf v}_i, \lfloor n/2 \rfloor) = \sum_{|J| = \lfloor n/2 \rfloor}
a_J ({\bf v}_i)^J, \qquad \forall 1 \leq i \leq m.
\]
This is a system $A_{m \times N} \cdot {\bf a} = 0$ of fewer-than-$N$
linear equations, in $N$ variables ${\bf a} = (a_J)_{|J| = \lfloor n/2
\rfloor}$. Therefore $\ker(A) \neq 0$ as desired, by classical results of
McCoy~\cite{McCoy}.\footnote{As this may not be universally known, and
was not explained in~\cite{KT}, we sketch a proof.
Write $A = [A_1 | A_2]$ with $A_1 \in R^{m \times m}$. First if
$\det(A_1) \neq 0$, then set ${\bf a}_2 := - (\det A_1) \mbm{1}_{N-m}
\neq 0$ and ${\bf a}_1 := {\rm adj}(A_1) A_2 \mbm{1}_{N-m}$, where ${\rm
adj}(A)$ denotes the adjugate matrix of $A$. Then $A \begin{pmatrix} {\bf
a}_1 \\ {\bf a}_2 \end{pmatrix} = 0$, whence $\ker(A) \neq 0$. Else if
$\det(A_1) = 0$, the Corollary in \cite{McCoy} yields ${\bf a}_1 \neq 0$
such that $A_1 {\bf a}_1 = 0$. Now set ${\bf a}_2 = 0$.}
\qed

\section{Special directions, via ballot sets}\label{Sballot}

By the above discussion following~\eqref{Evandermonde}, to complete the
proof of Theorem~\ref{T1} it remains to select $N$-many ``special''
directions ${\bf v}_i \in R^n$, each of them with at least two nonzero
coordinates (we will use here that $n \ge 3$ and so $\lceil n/2 \rceil
\ge 2$), such that the conditions
\[
0 = c({\bf v}_i, k) = \sum_{|J| = k} a_J ({\bf v}_i)^J, \qquad \forall 1
\leq i \leq N = \binom{n}{\lfloor n/2 \rfloor}
\]
imply $a_J = 0$ for all $2 \leq k \leq n$ and all $|J|=k$ (over the
ground ring $R$ in Theorem~\ref{T1}).

Remarkably, one can produce such vectors ${\bf v}_i$ \textit{with all
coordinates either $0$ or $1$} -- and which work at once in
\textit{every} unital commutative ground ring $R$. Translating into the
supports $S_i := \{ j \in [n] : ({\bf v}_i)_j = 1 \}$ of these ${\bf
v}_i$ produces a linear system of $\binom{n}{k}$ equations in as many
variables:
\begin{equation}\label{EdefineAnk}
A_{n,k} {\bf a}^{(k)} = 0, \qquad \text{where }
{\bf a}^{(k)} = (a_J)_{|J|=k} \text{ and the entry } (A_{n,k})_{i,J} =
\mbm{1}(S_i \supseteq J)
\end{equation}
for all $i \in \left[ \binom{n}{k} \right]$ and $J \in \binom{[n]}{k}$.

It thus suffices to produce $N$-many such sets $S_i$ such that $A_{n,2},
\dots, A_{n,n}$ (defined in~\eqref{EdefineAnk}) are nonsingular matrices.
More strongly, we will provide an explicit construction of such sets
$S_i$, which are ``universal'' in that they work simultaneously for every
unital commutative $R$, for all $0 \leq k \leq n$. In particular, using
$R = \mathbb{Z} / q \mathbb{Z}$ for all integers $q$ suggests that $\det
A_{n,k}$ would need to be $\pm 1$ for all $k \leq n$.

\subsection{Ballot sets and unimodular $0$-$1$ matrices}

We now indeed construct such ``universal'' 0-1 matrices, starting with
some notation. Recall that a square integer matrix is \emph{unimodular}
if its determinant is $\pm 1$. We will only care about the structure of
matrices up to independently reordering the rows and the columns, since
that maintains unimodularity. We will denote such equivalence of matrices
by $C_1\simeq C_2$.

\begin{definition}[Ballot set]\label{def:ballot}
We say that a set $B = \{x_1,\ldots,x_k\}\subseteq[n]$ with $x_1 < \cdots
< x_k$ is a \emph{ballot set} if $x_i \ge 2i$ for all $1\leq i\leq k$.
(For the choice of name, see the proof of
Proposition~\ref{prop:ballot-counting}.)
By convention, the empty set is ballot. Also set $\binom{n}{-1} := 0$.

Let $\mc{B}_{n,k} = \{B\in\binom{[n]}{k}\colon B\text{ is ballot}\}$.
Note that $\mc{B}_{n,k}\neq\emptyset$ implies $0 \leq k \leq \lfloor
n/2\rfloor$.
Let $\mc{S}_{n,k}=\{[n]\setminus B\colon B\in\mc{B}_{n,k}\}$. Let
$\mc{B}_{n,\leq k}=\bigsqcup_{i=0}^k\mc{B}_{n,i}$ and similarly for
$\mc{S}_{n, \leq k}$.
\end{definition}

\begin{proposition}\label{prop:ballot-counting}
For $n\ge 1$ and $0 \leq k \leq \lfloor n/2\rfloor$, we have
$|\mc{S}_{n,k}|=|\mc{B}_{n,k}|=\binom{n}{k}-\binom{n}{k-1}$.
\end{proposition}
\begin{proof}
This is a well-known consequence of the diagonal reflection argument,
used in Bertrand's ballot theorem with ties allowed.
\end{proof}

The following result will complete the proof of Theorem~\ref{T1}.

\begin{theorem}\label{thm:main}
Given $n\ge 1$, consider all ballot sets $B_1, \ldots,
B_N\subseteq[n]=\{1,\ldots,n\}$ in increasing order of size, with ties
broken arbitrarily. Let $S_i = [n]\setminus B_i$ for all $1 \leq i \leq
N$. For $0 \leq k \leq n$, consider the matrix $A_{n,k}$ with rows
indexed by $I\in\{S_i\colon 1 \leq i \leq \binom{n}{k}\}$ and columns
indexed by $J\in\binom{[n]}{k}$, where $A_{n,k}(I, J) =
\mbm{1}[I\supseteq J]$. Then $\det A_{n,k}\in\{\pm 1\}$.
\end{theorem}

Note that by Proposition~\ref{prop:ballot-counting} we have
$N=\binom{n}{\lfloor n/2\rfloor}$, so the matrices $A_{n,k}$ are in fact
well-defined. Additionally, Proposition~\ref{prop:ballot-counting}
implies that the rows of $A_{n,k}$ are in fact indexed by $\mc{S}_{n,
\leq k}$ when $k \leq \lfloor n/2\rfloor$, and by $\mc{S}_{n, \leq n-k}$
when $k > n/2$.

\subsection{Ballot matrices, triangularity, and unimodular transforms}

We now define the key matrices that we will inductively prove are
unimodular.

\begin{definition}[Ballot matrices]
For integers $k,n$ with $0 \leq k \leq \lceil n/2 \rceil$, we define
\emph{ballot matrices} $C_{n,k}, D_{n,k}$ with rows indexed by
$\mc{B}_{n, \leq k}$ and columns indexed by $\binom{[n]}{k}$ with
\[
C_{n,k}(I, J) = \mbm{1}[I\subseteq J],\quad D_{n,k}(I, J) = \mbm{1}[I\cap
J = \emptyset.]
\]
\end{definition}

Note that by replacing $I$ in the definition of $A_{n,k}$ by
$[n] \setminus I$, and by replacing $J$ by $[n] \setminus J$ in the case
$k>n/2$, we have
\[
A_{n,k}\simeq\begin{cases}D_{n,k}&\text{for } k \leq n/2,\\
C_{n,n-k}&\text{for }k>n/2,\end{cases}
\]
and also $C_{n,k}, D_{n,k}$ are square for $0 \leq k \leq n/2$.

Therefore, to show Theorem~\ref{thm:main} it suffices to prove:

\begin{proposition}[Ballot matrices are unimodular]\label{prop:ballot-unimodular}
If $0 \leq k \leq n/2$, then $\det C_{n,k},\det D_{n,k}\in\{\pm1\}$.
\end{proposition}

The two key lemmata used in the proof are the following:

\begin{lemma}[Unimodular reflection of ballot matrices]\label{lem:ballot-reflection}
For $0\leq k \leq n/2$, there is a matrix $T_{n,k}$ with $\det
T_{n,k}\in\{\pm 1\}$ such that $D_{n,k}=T_{n,k}C_{n,k}$. In fact, given
$I,I'\in\mc{B}_{n, \leq k}$ we may write
\[
T_{n,k}(I, I') = (-1)^{|I'|}\mbm{1}[I'\subseteq I].
\]
\end{lemma}

\begin{proof}
For all $I\in\mc{B}_{n, \leq k}$ and $J\in\binom{[n]}{k}$ we have:
\begin{align*}
& \sum_{I'\in\mc{B}_{n, \leq k}} (-1)^{|I'|}\mbm{1}[I'\subseteq
I]\cdot\mbm{1}[I'\subseteq J]\\
&= \sum_{\substack{I'\in\mc{B}_{n, \leq k}\\I'\subseteq I, \,
I' \subseteq J}} (-1)^{|I'|}\mbm{1}[I'\subseteq I] \cdot
\mbm{1}[I'\subseteq J]\\
&= \sum_{I'\subseteq I \cap J}(-1)^{|I'|}\mbm{1}[I'\subseteq
I]\cdot\mbm{1}[I'\subseteq J] = \sum_{I'\subseteq
I \cap J}(-1)^{|I'|}\mbm{1}[I'\subseteq I\cap J]\\
&=\mbm{1}[I\cap J = \emptyset].
\end{align*}

The first and third equalities are clear. The second is due to the fact
that $\mc{B}_{n,\leq k}$ is clearly closed under taking subsets. The
fourth is by the principle of inclusion-exclusion.

If we sort the axes in increasing order of size of $I'$ and $I$, with
ties broken arbitrarily, then the matrix $T_{n,k}$ is upper triangular
with $\pm1$ on the diagonal. So $T_{n,k}$ is unimodular.
\end{proof}

\begin{lemma}[Triangular decomposition of ballot matrices]\label{lem:ballot-triangle}
For $1\leq k\leq n/2$, we may reorder the axes of $C_{n,k}$ such that it
is block lower triangular, where the $2$ diagonal blocks are
$C_{n-1,k-1}$ and $C_{n-1,k}$.
\end{lemma}

\begin{proof}
Recall that rows of $C_{n,k}$ are indexed by $I\in\mc{B}_{n,\leq k}$, and
columns by $J\in\binom{[n]}{k}$. We put the rows where $I\ni n$ at the
top and $I\not\ni n$ at the bottom, and put the columns where $J\ni n$ at
the left and $J\not\ni n$ at the right, forming $2\times 2 = 4$ blocks.
Since $C_{n, k}(I, J) = \mbm{1}[I\subseteq J]$, the block lower
triangular structure is clear.

The lower right block with $n\notin I\cup J$ clearly corresponds to
$C_{n-1,\leq k}$. The upper left block with $n\in I\cap J$ corresponds to
$C_{n-1,\leq k-1}$ by deleting $n$, which respects the necessary ballot
and containment properties. 
\end{proof}

Note that in the case $n=2k$, the matrix $C_{n-1,\leq k}$ does not satisfy
$k\leq (n-1)/2$, so we must be careful with respect to this case when using
induction. Our final lemma lets us handle this case:
\begin{lemma}\label{lem:ballot-complement}
For $k\ge 1$, we have $C_{2k-1,k}\simeq D_{2k-1,k-1}$.
\end{lemma}

\begin{proof}
First, note that $\mc{B}_{2k-1,k} = \emptyset$ hence $\mc{B}_{2k-1, \leq
k} = \mc{B}_{2k-1, \leq k-1}$. Second, replace the columns indexed by
$J\in\binom{[2k-1]}{k}$ by $J'=[2k-1]\setminus J$. For $I\in\mc{B}_{2k-1,
\leq k}$ this transforms the condition $I\subseteq J$ into the condition
$I\cap J'=\emptyset$. We are done.
\end{proof}

Now the finish is easy.

\begin{proof}[Proof of Proposition~\ref{prop:ballot-unimodular}]
Let us first consider the case $k=0$. Note that $C_{n,0}$ and $D_{n,0}$
are $1\times 1$ matrices with element $1$, so it is clear in this case.
Also, note that $n=1$ follows since the only possible $k$ is $k=0$.

Now, perform induction on $n$ and assume $n\ge 2$. If $k=0$ we are done,
hence assume $k\ge 1$. Reduce unimodularity of $D_{n,k}$ to that of
$C_{n,k}$ by Lemma~\ref{lem:ballot-reflection}. Reduce unimodularity of
$C_{n,k}$ to that of $C_{n-1,k-1}$ and $C_{n-1,k}$ by
Lemma~\ref{lem:ballot-triangle}.

We always have $n-1\ge 2(k-1)$, so we may apply induction for the matrix
$C_{n-1,k-1}$. When $n-1\ge 2k$, we may apply it for $C_{n-1,k}$ to
finish. Otherwise, $n = 2k$. In this case we apply
Lemma~\ref{lem:ballot-complement} to reduce unimodularity of
$C_{n-1,k}=C_{2k-1,k}$ to that of $D_{2k-1,k-1}$, which follows from
induction. We are done.
\end{proof}

\begin{remark}
After writing this paper, the first-named author was reminded by Arvind
Ayyer that they had discussed the question in~2023! And Ayyer had
conjectured that the following collection of sets works:
\begin{equation}\label{Earvind}
I = [n] \setminus \{ a_1 < \cdots < a_k \}, \quad \text{satisfying:}
\quad 1 \leq a_i \leq n - 2k - 1 + 2i\ \forall i \in [k],
\end{equation}
where $k = 0, 1, 2, \dots$.
Using ballot sets, we now explain why the family~\eqref{Earvind} indeed
works. First note that one can take the complements of other sets than
ballot sets to obtain $N$-many directional vectors -- say applying a
common permutation $w$ to all ballot subsets. Apply the longest
permutation (aka type $A$ Weyl group word) $w_\circ := (i \leftrightarrow
n+1-i)$ to the above family, to get $a'_i := n+1 - a_i$ satisfying the
conditions
\[
n \ge a'_1 > \cdots > a'_k \ge 1, \quad a'_i \ge 2k-2i+2\ \forall i \in
[k].
\]
We next write these elements in reverse order as $b_i := a'_{k+1-i}$, and
have thus transformed the conjectural family of subsets into:
\[
[n] \setminus \{ b_1 < \cdots < b_k \}, \quad \text{satisfying:}
\quad b_i \ge 2i\ \forall i \in [k].
\]
But these are precisely the complements of ballot sets, and so the $I$
in~\eqref{Earvind} also work.
\qed
\end{remark}

\section{Proofs of Theorems~\ref{T2} and~\ref{T3} over rings in
$\mathfrak{R}$}\label{SthmB}

Section~\ref{Sballot} concludes the proof of Theorem~\ref{thm:main} and
so of Theorem~\ref{T1}. We now turn to Theorem~\ref{T2}; before proving
it, we explain the motivation in working over not one field but a product
of them. This is because even over one field, the proof of
Theorem~\ref{T2} crucially uses the fact that for a prime power $q \leq
n$, the power function $r \mapsto r^q$ is the identity on $\F_q$, yet has
exponent in $[2,n]$. It is now natural to ask for which (other)
commutative unital rings $R$ does such an exponent $>1$ exist.
The following result answers this question for finite rings and provides
the above motivation.

\begin{proposition}\label{Ppower}
Given a finite unital commutative ring $R$, the following are equivalent.
\begin{enumerate}
\item There exists $k > 1$ such that $r^k = r$ for all $r \in R$.

\item $R$ has no (nonzero) nilpotents, i.e., its nilradical is the
trivial ideal $(0)$.

\item $R$ is a direct product of (finite) fields.
\end{enumerate}
\end{proposition}

\begin{proof}
Suppose $r \in R$ and $r^m=0$ for some $m>1$. Then $s := r^{m-1}$
satisfies: $s^2 = 0$, whence $s^k = 0$. By~(1), $s=0$, so $r^{m-1}=0$. If
$m=2$ then we are done, else repeat this argument with $m-1$ in place of
$m$. This shows $(1) \implies (2)$.

If~(2) holds, then $R$ is reduced, as well as Artinian. Hence by the
structure theorem for Artinian rings (see e.g.\
\cite[Theorem~8.7]{AMcD}), it is a direct product of local Artinian rings
with no nilpotents, i.e.\ of fields. This shows~(3).

Finally, suppose the finite ring $R = \times_{i=1}^l \F_i$, with $|\F_i|
=: q_i < \infty$ for all $i$. Let $L := {\rm lcm}(q_1-1, \dots, q_l-1)$.
Then $r^L = 1$ for all $i$ and all $r \in \F_i^\times$. Hence
\[
((r_i)_{i=1}^l)^{L+1} = (r_i^{L+1})_i = (r_i)_i, \qquad \forall 1 \leq i
\leq l, \ r_i \in \F_i
\]
which shows~(1) and concludes the proof.
\end{proof}

Proposition~\ref{Ppower} motivated this section; we remark for
completeness that it holds in greater generality:

\begin{proposition}
Suppose $k > 1$ is an integer and $R = \times_{i \in I} R_i$, where each
$R_i$ is an integral domain or a local ring with nilpotent maximal ideal.
Then $r \mapsto r^k$ is the identity map on $R$ if and only if every
$R_i$ is a finite field satisfying: $|R_i| - 1$ divides $k-1$.
\end{proposition}

\noindent The proof is similar to that of Proposition~\ref{Ppower}, hence
omitted.

It remains to show Theorem~\ref{T2}, and that Noetherian rings fall in
our framework. We start with the latter:

\begin{proof}[Proof of Proposition~\ref{Pnoetherian}]
Throughout this proof, $R$ is Noetherian. We quickly recall some standard
facts.
(a)~An associated prime is $\mathfrak{p} \in {\rm Spec}(R)$ such that
$\mathfrak{p} = {\rm Ann}_R(r)$ for some $0 \neq r \in R$.
(b)~$R$ has finitely many associated primes ${\rm Ass}(R)$, and their
union is precisely the zerodivisor set $ZD(R)$.

We now proceed. Recall \cite[Proposition~3.11(iv)]{AMcD} that
prime ideals in $Q(R)$ are precisely $(ZD(R)^c)^{-1} P$ for some prime
ideal avoiding $ZD(R)^c$. Thus if $\mfm' \subset Q(R)$ is maximal, then
$\mfm' = Q(R) \mfm = (ZD(R)^c)^{-1} \mfm$ for some maximal ideal $\mfm
\subseteq ZD(R) = \bigcup_{\mathfrak{p} \in {\rm Ass}(R)} \mathfrak{p}$.
By prime avoidance \cite[Proposition~1.11]{AMcD}, $\mfm$ is in some
associated prime $\mathfrak{p}$ of $R$. Hence by definition (see (a)
above), there exists $0 \neq r \in R$ such that $r \mfm = 0$, whence
$\frac{r}{1} \cdot \mfm' = \frac{r}{1} \mfm Q(R) = 0$. Finally,
$\frac{r}{1} \neq 0$ in $(ZD(R)^c)^{-1} R$ since $r \neq 0$.
\end{proof}

We end by proving Theorem~\ref{T2}. This requires an intermediate
technical step.

\begin{proposition}\label{Pgemini}
Suppose $n \ge 2$ and $R$ is a unital commutative ring satisfying
Definition~\ref{Dkasch}(2): $Q(R)$ is Kasch. Then the following are
equivalent:
\begin{enumerate}
\item $R$ is $n$-Vandermonde full: ${\rm Ann}_R (I_{p_n}) = 0$.

\item For all maximal ideals $\mfm$ of $Q(R)$, we have $|R / (\mfm \cap
R)| \ge n+1$.
\end{enumerate}
Here we note that $\mfm \cap R$ is a prime ideal of $R$.
\end{proposition}

\begin{proof}
First suppose~(1) holds, and let $\mfm$ be maximal in $Q(R)$. Set $P :=
\mfm \cap R$; then $P$ is a prime ideal of $R$ because if $xy \in P$ for
$x,y \in R$ then $\frac{x}{1} \frac{y}{1} \in \mfm$, which is prime in
$Q(R)$. Now $\mfm$ has some nonzero annihilator $\frac{a}{s}$ with $a \in
R$ and $s \in ZD(R)^c$; so $a P = 0$ in $R$.

Suppose for contradiction that $|R / P| \leq n$. Let $\pi : R
\twoheadrightarrow R / P$ be the quotient map, and compute:
\[
\pi(p_n({\bf r})) = p_n(\pi(r_1), \dots, \pi(r_n)) = 0, \qquad \forall
{\bf r} \in R^n.
\]
This yields $I_{p_n} \subseteq P$, and so $a \cdot I_{p_n} = 0$ in $R$,
contradicting~(1). This proves one implication.

Conversely, suppose for contradiction that~(2) holds but $a \cdot I_{p_n}
= 0$ in $R$ for some $a \neq 0$. Then $Q(R) \cdot I_{p_n}$ is an ideal in
$Q(R)$ that is killed by $\frac{a}{1}$, so it is a proper ideal and hence
in some maximal ideal $\mfm$. Set the prime ideal $P := \mfm \cap R$ and
$\pi : R \twoheadrightarrow R / P$ as above; then $I_{p_n} \subseteq P$.
Choose distinct nonzero elements $t_1, \dots, t_n \in R/P$ by hypothesis,
and lift them to $s_1, \dots, s_n \in R$; then $\pi(p_n({\bf s})) =
p_n({\bf t}) \neq 0$, since $R / P$ is an integral domain. This
non-equality contradicts $p_n({\bf s}) \in I_{p_n} \subseteq P$.
\end{proof}

Finally, we have:

\begin{proof}[Proof of Theorem~\ref{T2}]
Clearly $(1) \implies (2) \implies (3)$, and the assertion in~(a) follows
from (the proof in Section~\ref{Sballot} of) Theorem~\ref{T1}. That $N$
is best possible was shown in Section~\ref{Smccoy}.

Coming to the proof of~(b), suppose $R = \times_{i \in I} R_i \in
\mathfrak{R}$, with $Q(R_i)$ Kasch for all $i$.
We first claim that
\begin{equation}\label{Eclaim}
{\rm Ann}_R (I_{p_n}(R)) = \times_{i \in I} {\rm Ann}_{R_i}
(I_{p_n}(R_i)).
\end{equation}

Indeed, suppose $(r_i)_{i \in I}$ is in the LHS, and consider any $i_0
\in I$ and ${\bf s}_{i_0} \in R_{i_0}^n \hookrightarrow R^n$ (the
canonical embedding of abelian groups here). Then $(r_i)_i \cdot p_n({\bf
s}_{i_0}) = 0$ for all ${\bf s}_{i_0}$, so taking the $i_0$-component
gives $r_{i_0} \cdot p_n(R_{i_0}^n) = 0$. This holds for all $i_0$, so
$(r_i)_i$ is in the RHS. This proves one inclusion.

Conversely, let $r_i \in {\rm Ann}_{R_i}(I_{p_n}(R_i))$ for all $i$.
Choose finitely many elements ${\bf s}_{i1}, \dots, {\bf s}_{ik} \in
R_i^n$ for each $i$, and scalars $t_1 = (t_{i1})_i, \dots, t_k =
(t_{ik})_i \in R$, to write a general element $r \in I_{p_n}(R)$ as:
\[
r = \sum_{j=1}^k t_j \cdot p_n(({\bf s}_{ij})_i) \in I_{p_n}(R),
\]
via the canonical isomorphism $: \times_{i \in I} R_i^n \to R^n$.
Then $(r_i)_i \cdot r$ can be written in terms of its individual
components -- the $i$th component is
$r_i \cdot \sum_{j=1}^k t_{ij} p_n({\bf s}_{ij})$, with $p_n({\bf s}_{ij})
\in I_{p_n}(R_i)$. Hence $r_i \cdot p_n({\bf s}_{ij}) = 0$ for all $i,j$,
and so $(r_i)_i$ is in the LHS. \medskip

Having proved~\eqref{Eclaim}, we finish the proof of Theorem~\ref{T2}(b).
If $R = \times_i R_i \in \mathfrak{R}$ is not $n$-Vandermonde full, then
by~\eqref{Eclaim} there exists $i_0 \in I$ such that ${\rm
Ann}_{R_{i_0}}(I_{p_n}(R_{i_0})) \neq 0$. By Proposition~\ref{Pgemini},
there exists a maximal ideal $\mfm$ of $Q(R_{i_0})$ such that $\F :=
R_{i_0} / (\mfm \cap R_{i_0})$ has size at most $n$. But we already saw
that $P := \mfm \cap R_{i_0}$ is prime, so $\F$ is an integral domain,
whence a finite field of size $q \leq n$, say. Moreover, the Kasch
property of $Q(R_{i_0})$ yields $0 \neq b \in R_{i_0}$ and $s \in R_{i_0}
\setminus ZD(R_{i_0})$ such that $\frac{b}{s} \cdot \mfm =  0$. It
follows that $b \cdot P = 0$.

We now provide an explicit multi-affine polynomial satisfying~(2) in
Theorem~\ref{T2} but not~(1). Namely, consider any polynomial
\[
f({\bf x}) := a_\emptyset + \sum_{l \ge 0} \sum_{J \subseteq [n] : |J|=
1 + l(q-1)} a_J {\bf x}^J,
\]
with scalars $a_\emptyset \in R$ and $a_J \in {\rm Ann}_{R_{i_0}}(P)
\subseteq R_{i_0}$ for all $J \neq \emptyset$ (here and below, we
identify $R_{i_0}$ with its copy in $R_{i_0} \times (\times_{i \in I
\setminus \{i_0\}} (0)) \subset R$ as above), such that $a_J \neq 0$ for
at least one $J= J_0$ with $|J_0| \ge q$. (For instance, $f({\bf x}) = b
x_1 \cdots x_q$ with $b$ as above.) This is not affine-linear given the
monomial $a_{J_0} {\bf x}^{J_0}$ in it; and yet, given any ${\bf v} \in
R^n$, we compute:
\[
f(r \cdot {\bf v}) = a_\emptyset + \sum_{l \ge 0} \sum_{J \subseteq [n]
: |J|= 1 + l(q-1)} a_J {\bf v}^J r^{|J|}, \qquad \forall r \in R =
\times_i R_i.
\]

Since $|R_{i_0} / P| = q$, note that
\[
\pi_{i_0}(r^q - r) = \pi_{i_0}(r)^q - \pi_{i_0}(r) \in \ker_{R_{i_0}}
(\eta_{i_0}) = P, \qquad \forall r \in R
\]
where $\pi_{i_0} : R \twoheadrightarrow R_{i_0}$ is the projection, and
$\eta_{i_0} : R_{i_0} \twoheadrightarrow R_{i_0} / P = \F$ is the
quotient map. Use this repeatedly, along with that $a_J \in {\rm
Ann}_{R_{i_0}}(P)$ -- so that $a_J \pi_{i_0}(r^q) = a_J \pi_{i_0}(r)$ --
to obtain:
\[
a_J r = a_J \pi_{i_0}(r) = a_J \pi_{i_0}(r^q) = a_J r^q = \cdots = a_J
r^{2q-1} = \cdots.
\]
Thus, $f(r \cdot {\bf v}) = a_\emptyset + r \cdot (f({\bf v}) -
a_\emptyset)$, and so $f|_{R{\bf v}}$ is affine-linear.
\end{proof}

We end with:

\begin{proof}[Proof of Theorem~\ref{T3}]
Theorem~\ref{T1} shows $(1) \implies (3)$. 
Proposition~\ref{Pgemini} and Equation~\eqref{Eclaim} show the
contrapositive of $(2) \Longleftrightarrow(1)$. Finally, the proof of
Theorem~\ref{T2} shows the contrapositive of $(3) \implies (2)$.
\end{proof}

\subsection*{Acknowledgements}

The unimodularity reflection result in Lemma~\ref{lem:ballot-reflection}
was suggested by Claude Opus~4.8. Example~\ref{Exgemini} and the proofs
of Propositions~\ref{Pnoetherian} and~\ref{Pgemini} were suggested by
Google Gemini Pro~3.1.
A.K.\ thanks Prahlad Harsha, Jaikumar Radhakrishnan, and Akaki Tikaradze
for useful discussions about complexity theory and ring theory, and
acknowledges support from the Government of India via an ARG grant
(ANRF/ARG/2025/011665/MS) from ANRF, a Shanti Swarup Bhatnagar Award from
CSIR, and the DST FIST program-2021 TPN-700661.



\end{document}